\documentclass[a4paper]{amsart}
\usepackage[margin=3.8cm]{geometry}
\usepackage{graphicx}
\usepackage{graphics}
\usepackage{amsmath}
\usepackage{latexsym}
\usepackage{amssymb}
\usepackage[all]{xy}
\usepackage{psfrag}
\usepackage[tight]{subfigure}
\xyoption{matrix}
\xyoption{arrow}
\usepackage{color}
\usepackage{float}

\begin{document}

\renewcommand{\th}{\operatorname{th}\nolimits}
\newcommand{\rej}{\operatorname{rej}\nolimits}
\newcommand{\extto}{\xrightarrow}
\renewcommand{\mod}{\operatorname{mod}\nolimits}
\newcommand{\ul}{\underline}
\newcommand{\Sub}{\operatorname{Sub}\nolimits}
\newcommand{\ind}{\operatorname{ind}\nolimits}
\newcommand{\Fac}{\operatorname{Fac}\nolimits}
\newcommand{\add}{\operatorname{add}\nolimits}
\newcommand{\soc}{\operatorname{soc}\nolimits}
\newcommand{\Hom}{\operatorname{Hom}\nolimits}
\newcommand{\shape}{\mathcal{S}}
\newcommand{\Rad}{\operatorname{Rad}\nolimits}
\newcommand{\RHom}{\operatorname{RHom}\nolimits}
\newcommand{\uHom}{\operatorname{\underline{Hom}}\nolimits}
\newcommand{\End}{\operatorname{End}\nolimits}
\renewcommand{\Im}{\operatorname{Im}\nolimits}
\newcommand{\Ker}{\operatorname{Ker}\nolimits}
\newcommand{\Coker}{\operatorname{Coker}\nolimits}
\newcommand{\Ext}{\operatorname{Ext}\nolimits}
\newcommand{\op}{{\operatorname{op}}}
\newcommand{\Ab}{\operatorname{Ab}\nolimits}
\newcommand{\id}{\operatorname{id}\nolimits}
\newcommand{\pd}{\operatorname{pd}\nolimits}
\newcommand{\A}{\operatorname{\mathcal A}\nolimits}
\newcommand{\C}{\operatorname{\mathcal C}\nolimits}
\newcommand{\D}{\operatorname{\mathcal D}\nolimits}
\newcommand{\B}{\operatorname{\mathcal B}\nolimits}
\newcommand{\R}{\operatorname{\mathcal R}\nolimits}
\newcommand{\M}{\operatorname{\mathcal M}\nolimits}
\newcommand{\X}{\operatorname{\mathcal X}\nolimits}
\newcommand{\Y}{\operatorname{\mathcal Y}\nolimits}
\newcommand{\F}{\operatorname{\mathcal F}\nolimits}
\newcommand{\Z}{\operatorname{\mathbb Z}\nolimits}
\renewcommand{\P}{\operatorname{\mathcal P}\nolimits}
\newcommand{\T}{\operatorname{\mathcal T}\nolimits}
\newcommand{\G}{\Gamma}
\renewcommand{\L}{\Lambda}
\newcommand{\bdot}{\scriptscriptstyle\bullet}
\renewcommand{\r}{\operatorname{\underline{r}}\nolimits}
\newtheorem{theorem}{Theorem}[section]
\newtheorem{prop}[theorem]{Proposition}
\newtheorem{corollary}[theorem]{Corollary}
\newtheorem{lemma}[theorem]{Lemma}

\newcommand{\new}{\textcolor{blue}}
\newcommand{\old}{\textcolor{red}}

\theoremstyle{definition} 
\newtheorem{remark}[theorem]{Remark}
\newtheorem{definition}[theorem]{Definition}
\newtheorem{example}[theorem]{Example}
\newtheorem{proposition}[theorem]{Proposition}

\title[A circular order on edge-coloured trees]{A circular order on edge-coloured trees and RNA $m$-diagrams}

\author[Marsh]{Robert J. Marsh}
\address{School of Mathematics \\
University of Leeds \\
Leeds LS2 9JT \\
United Kingdom}
\email{marsh@maths.leeds.ac.uk}

\author[Schroll]{Sibylle Schroll}
\address{
Department of Mathematics \\
University of Leicester \\
University Road \\
Leicester LE1 7RH \\
United Kingdom}
\email{ss489@le.ac.uk}


\begin{abstract}
We study a circular order on labelled, $m$-edge-coloured trees with
$k$ vertices, and show that the set of such trees with a fixed circular order is
in bijection with the set of RNA $m$-diagrams of degree $k$, combinatorial objects
which can be regarded as RNA secondary structures of a certain kind.
We enumerate these sets and show that the set of trees with a fixed circular
order can be characterized as an equivalence class for the transitive closure of
an operation which, in the case $m=3$, arises as an induction in the context
of interval exchange transformations.
\end{abstract}

\thanks{ \vskip 0.1cm
This work was supported by the Engineering and Physical Sciences
Research Council [grant number EP/G007497/1], the Mathematical Sciences
Research Institute (MSRI) in Berkeley, California, and by the Leverhulme
Trust in the form of an Early Career Fellowship for Sibylle Schroll.}

\date{22 September 2013}

\maketitle

\section{Introduction}

Interval exchange transformations are, roughly speaking, a generalization of a rotation of a circle. More precisely, given a  partition of the unit interval into smaller segments, an interval
 exchange transformation is a map from the unit interval to itself  where the segments are rearranged according to a choice of permutation.
 Recently there has been particular interest in the class of interval exchange transformations induced by the permutation $(k\ k-1\ \cdots \  1)$. In \cite{cfz11} a combinatorial interpretation
  in terms of an  induction proccess on labelled trees with edges coloured with $3$  possible colours, called trees of relations, has proven a fruitful  tool leading to new results in the
   understanding of the associated languages. In this paper we extend this combinatorics by defining a generalization of the induction process to the case of labelled trees with edges coloured
    with $m$ possible colours and study its properties.

Given positive integers $k$ and $m$, an \emph{$m$-edge-coloured} tree with $k$ vertices is a tree whose edges are
labelled with $m$ possible colours in such a way that no vertex is incident
with two edges of the same colour. It is said to be \emph{labelled} if its
vertices are labelled with $\{1,2,\ldots ,k\}$.
Such trees were shown to possess a circular order in the case $m=3$ in~\cite{cfz11} corresponding to the permuation of the interval exchange transformation.
By giving a new proof of this fact, we show that the circular order can be generalized to
arbitrary $m$ and we study its relation to the generalized induction proccess we define.

Firstly, we show that the set $\T_{k,m}$ of labelled $m$-edge-coloured trees
with $k$ vertices and a fixed circular order is in bijection with a
collection of combinatorial objects consisting of collections of noncrossing
arcs in a disk with coloured marked points on its boundary, satisfying a certain
matching rule. Such diagrams can be regarded as generalized
RNA secondary structures of a certain kind, in the sense of~\cite{waterman78a},
drawn in the Nussinov circle representation~\cite{npgk78}, so we refer to them as
\emph{RNA m-diagrams of degree $k$}.
Such a bijection was given for the case $m=3$ in~\cite{cfz11}.

We compute the cardinality of $\T_{k,m}$ (and thus also the number of RNA $m$-diagrams
of degree $k$) using the well-known correspondence between trees and $m$-angulations of
a polygon which induces a bijection with a collection of appropriately labelled $m$-angulations

We prove the existence of a transformation of labelled $m$-edge-coloured
trees preserving the circular order. We call this transformation
\emph{generalized induction} as it generalizes the induction in the case $m=3$,
studied combinatorially in~\cite{cfz11}.
In~\cite{ferenczizamboni10} it was shown that this induction (for $m=3$)
gives rise to languages naturally generalizing Sturmian languages and
arising from interval exchange transformations.
We also give an example showing that the most straightforward generalization
of this induction to the case $m>3$ does not in general preserve the circular order.

We go on to show that the equivalence classes of the transitive closure of
the generalized induction are characterized by the circular order.
Thus the set of equivalence classes is in bijection with the set of
$k$-cycles in the symmetric group of degree $k$.

Finally, motivated by snake-triangulations and $m$-snakes in cluster theory, we give an interpretation of
the generalized induction process in terms of labelled $m$-angulations of polygons.

This article is structured as follows. In Section 2, we introduce the
combinatorial objects considered in the paper. In Section 3 we give the
bijections mentioned above and compute the cardinality of the set of
labelled $m$-edge-coloured trees with a fixed circular order. In Section
4 we show that the circular order characterizes the equivalence classes
of the transitive closure of the generalized induction and in section 5
we show that induction can be given by a composition of flips of diagonals in an $m$-angulation.


\section{Some combinatorial objects}

We first introduce the main combinatorial objects we will be considering:
edge-coloured trees, RNA $m$-diagrams, and $m$-angulations.

Given positive integers $k,m$, we consider $m$-edge-coloured trees
on $k$ vertices, i.e.\ trees whose edges are coloured
with one of the $m$ symbols $S_1,S_2,\ldots ,S_m$ in such a way that
no two edges incident with the same vertex have the same colour. We say
that such a tree is \emph{labelled} if its vertices are labelled with
$\{1,2,\ldots ,k\}$. As usual, if the connectedness assumption is not
satisfied, we refer to the corresponding objects as forests. See
Figure~\ref{f:bijectionexample1}(b) for an example of a (rooted)
$m$-edge-coloured tree and (c) for an example of a labelled
$m$-edge-coloured tree.

Given an $m$-edge coloured tree $\shape$ with $k$ vertices, each symbol
$S_r$ determines a map (with the same name) from the set of vertices
of $\shape$ to itself. A vertex is fixed by $S_r$ unless it is incident with an
edge coloured $S_r$, in which case it is sent to the vertex at the other
end of that edge. Let $\sigma_{\shape}$ be the composition $S_mS_{m-1}\cdots
S_1$: this is a permutation of the vertices of $\shape$. We use the
same definition for a labelled $m$-edge coloured tree, $G$, obtaining
a permutation $\sigma_G$ in the symmetric group of degree $k$.
We refer to $\sigma_{\shape}$ (respectively, $\sigma_G$) as the
\emph{circular order} of $\shape$ (respectively, $G$). In the $3$-symbol case
these maps were considered in~\cite{cfz11}. The circular order
is always a $k$-cycle:

\begin{lemma} \label{l:alwaysacycle}
Let $G$ be a labelled $m$-edge-coloured tree with $k$ vertices.
Then the circular order $\sigma_G$ of $G$ is a $k$-cycle.
\end{lemma}

\begin{proof}
This is clearly true if $k=1$ or $2$. Suppose it holds for smaller
values of $k$. Let $v$ be a vertex of $G$ which is not a leaf.
Suppose that $v$ is incident with edges $e_1,e_2,\ldots ,e_d$ in
$G$, coloured with $S_{r_1},S_{r_2},\ldots ,S_{r_d}$
respectively, where $r_1<r_2<\cdots <r_d$. Let the end-points of
these edges (other than $v$) be $v_1,v_2,\ldots ,v_d$. Removing $v$
from $G$ leaves $d$ subtrees $G'_1,G'_2,\ldots ,G'_d$ incident with
$v_1,v_2,\ldots ,v_d$ respectively. Let $G_i$ be the subtree $G'_i$
with $v$ and $e_i$ reattached to $v_i$.

By the inductive hypothesis the $\sigma_{G_i}$ are all cycles. Hence, repeatedly
applying $\sigma_{G_1}$ to $v$ cycles through the vertices of $G_1$. Since
$\sigma_G=\sigma_{G_1}$ on all vertices of $G_1$ except $w=\sigma_{G_1}^{-1}(v)$,
repeatedly applying $\sigma_G$ also cycles through all the vertices of $G_1$.
Since $r_2$ is minimal such that $S_{r_2}$ is a symbol colouring an edge
incident with
$v$ with $r_2>r_1$, $\sigma_G(w)$ will lie in $G'_2$. In fact $\sigma_G(w)=\sigma_{G_2}(v)$,
since in $G_2$, $v$ is not incident with any edge with symbol other than
$S_{r_2}$. Repeatedly applying $\sigma_G$ then cycles through the vertices of $G'_2$ before
coming back to $\sigma_{G'_2}^{-1}(v)$. Repeating this argument, we see that repeatedly
applying $\sigma_G$ to $v$ first cycles through $G'_1$, then through $G'_2,G'_3,\ldots ,G'_d$
in order before eventually returning to $v$.
\end{proof}

Next, we define an \emph{RNA $m$-diagram of degree $k$} as follows.
We take $k$ vertices $1,2,\ldots ,k$, numbered clockwise around a
circle. Each vertex contains the symbols $S_1,S_2,\ldots ,S_m$ written
in clockwise order. A collection of arcs connect equal symbols at
different vertices. A symbol at a vertex can be incident to at most one
arc. Such a diagram is \emph{noncrossing} if it can be drawn in such
a way that there are no crossings between the arcs. We say that it
is \emph{connected} provided there is a path between any two
vertices (where moving between symbols at a vertex is allowed). See
Figure~\ref{f:bijectionexample1}(a) for an example of a connected
noncrossing RNA $4$-diagram of degree $10$.

Diagrams of this kind (but with different rules) have been considered
in~\cite{bischjones97} in the context of Fuss-Catalan algebras.
Such diagrams can be regarded as a certain kind of abstract RNA secondary structures in the sense of~\cite{waterman78a} (see also~\cite[\S5]{hss99}),
drawn in the Nussinov circle representation~\cite{npgk78}.

Finally, let $P$ be an $(m-2)k+2$-sided regular polygon. An
$m$-angulation of $P$ is a collection of diagonals dividing $P$ into
$k$ $m$-gons.
We say that an $m$-angulation of $P$ is \emph{diagonal-coloured}
if each diagonal in $P$ is coloured with a symbol from the set
$\{S_1,S_2,\ldots ,S_m\}$ in such a way that the colours on the sides of
each $m$-sided polygon in the $m$-angulation are $S_1,S_2,\ldots ,S_m$ in
clockwise order. We say that it is \emph{rooted} if there is a
distinguished $m$-sided polygon in the $m$-angulation. We say that it is
\emph{$m$-gon-labelled} if the $m$-gons are labelled $1,2,\ldots ,k$.

\section{Bijections} \label{s:bijections}
Our main aim in this section is to compute the number of
connected noncrossing RNA $m$-diagrams of degree $k$ and to give a
bijective proof that this is the same as the number of labelled
$m$-edge-coloured trees with $k$ vertices and circular order $(k\
k-1\ \cdots \ 1)$.

\begin{theorem} \label{t:pseudoknotbijection}
There is a bijection between the following sets:
\begin{enumerate}
\item[(I)]
The set of degree $k$ noncrossing RNA $m$-diagrams.
\item[(II)]
The set of labelled $m$-edge-coloured forests $G$ with $k$ vertices
such that, on writing $G=\sqcup_i G_i$ as a union of connected components,
we have the following:
\begin{enumerate}
\item[(a)] If $i\not=j$ and $a_1,a_2\in G_i$, $b_1,b_2\in G_j$, we cannot
have that $a_1>b_1>a_2>b_2$.
\item[(b)] If $a\in G_i$ for some $i$, then $\sigma_G(a)$ is the
maximal vertex of $G$ less than $a$ lying in $G_i$ (or, if no such
vertex exists, it is the largest vertex of $G$ lying in $G_i$).
\end{enumerate}
\end{enumerate}
\end{theorem}

\begin{proof}
Let $\Sigma$ be a noncrossing RNA $m$-diagram of degree $k$ on
$m$ symbols as in (I).

Let $G$ be the labelled $m$-edge-coloured graph with $k$ vertices
whose edges given by the arcs of $\Sigma$.
That is, there is an edge between vertices $i$ and $j$ of $G$
coloured with $S_k$ if and only if there is an arc in $\Sigma$ between
the instances of the symbol $S_k$ in vertices $i$ and $j$ in $\Sigma$.

\textbf{Claim:} $G$ is a forest.

We prove the claim. Suppose, for a contradiction, there is a cycle
$$\xymatrix{
a_1 \ar@{-}[r]^{S_{i_1}} & a_2 \ar@{-}[r]^{S_{i_2}} & \cdots
& a_{p-1} \ar@{-}[r]^{S_{i_{p-1}}} & a_p \ar@{-}[r]^{S_{i_p}} & a_1 }$$
in $G$, and thus a corresponding cycle in $\Sigma$. Without loss of generality, we may assume that $i_1<i_2$.
For vertices $a,b$, we denote by $(a,b)$ the
set of vertices $c$ of $\Sigma$ lying strictly clockwise of $a$ and strictly
anticlockwise of $b$.

Then $a_3\in (a_2,a_1)$, since the arc in $\Sigma$ corresponding to the
edge in $G$ between $a_2$ and $a_3$ cannot cross the arc in $\Sigma$
corresponding to the edge in $G$ between $a_1$ and $a_2$.

By assumption on $\Sigma$, $i_2\not=i_3$. If $i_2>i_3$, there can be
no path in $\Sigma$ from the symbol $S_{i_3}$ in vertex $a_3$ of
$\Sigma$ back to vertex $a_1$ of $\Sigma$
without crossings, a contradiction, hence $i_2<i_3$.

Repeating this argument, we see that, moving clockwise on $\Sigma$
from $a_1$ we meet vertices $a_2,a_3,\ldots ,a_p$ in order before
returning to $a_1$, and that $i_1<i_2<\cdots <i_p$. But then the
arc between $a_p$ and $a_1$ (on symbol $S_{i_p}$ crosses the arc
between $a_1$ and $a_2$ (on symbol $S_{i_1}$), since $i_1<i_p$;
see Figure~\ref{f:cycleargument}. Hence $G$ has no cycles,
and must be a tree. The claim is shown.

\begin{figure}
\begin{center}
\psfragscanon
\psfrag{a1}{$\mathbf{a_1}$}
\psfrag{a2}{$\mathbf{a_2}$}
\psfrag{a3}{$\mathbf{a_3}$}
\psfrag{ap}{$\mathbf{a_p}$}
\psfrag{S1}{$S_1$}
\psfrag{Sm}{$S_m$}
\psfrag{Si1}{$S_{i_1}$}
\psfrag{Si2}{$S_{i_2}$}
\psfrag{Si3}{$S_{i_3}$}
\psfrag{Sip}{$S_{i_p}$}
\psfrag{Sip1}{$S_{i_{p-1}}$}
\includegraphics[width=4.5cm]{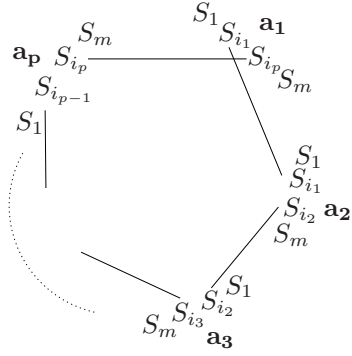} 
\end{center}
\caption{A cycle in $G$ leads to a crossing}
\label{f:cycleargument}
\end{figure}

We next prove that (a) holds. Suppose that $i\not=j$, $a_1,a_2\in G_i$,
$b_1,b_2\in G_j$, and $a_1>b_1>a_2>b_2$. Then $a_1,b_1,a_2,b_2$
follow each other anticlockwise around the circle. Then, since $a_1,a_2$
are in the same connected component of $G$, there is a path in $\Sigma$
from $a_1$ to $a_2$, and similarly from $b_1$ to $b_2$. The arrangement of
$a_1,a_2,b_1$ and $b_2$ implies that these two paths cross, a
contradiction. Hence no such arrangement can occur, and (a) is shown.

We next prove that (b) holds. It is enough to prove the following
claim:

\textbf{Claim:} Let $a\in G_i$. Then $\sigma_G(a)$ is the next vertex
of $G_i$ (considered as a vertex of $\Sigma$) anticlockwise on the circle
from $a$.

To prove the claim, we note that, by the definition of $\sigma_G$,
$\sigma_G(a)$ is connected to $a$ by a path in $G$:
$$\xymatrix{
a=a_1 \ar@{-}[r]^{S_{i_1}} & a_2 \ar@{-}[r]^{S_{i_2}} & \cdots & \cdots \ar@{-}[r]^{S_{i_{p-1}}} & a_{p}=\sigma_G(a)},$$
where $i_1<i_2<\cdots <i_{p-1}$. Furthermore, $a_1$ is not incident with
any arc with symbol $S_r$ for $r<i_1$, $a_p$ is not incident with any
arc with symbol $S_r$ for $r>i_{p-1}$, and, for $2\leq j\leq p-1$, $a_j$ is
not incident with any symbol $S_r$ for $i_{j-1}<r<i_j$.

The existence of the above path implies that $\sigma_G(a)$ lies in
$G_i$. Since $i_1<i_2<\cdots <i_p$ and there are no crossings, the path
must go clockwise around the circle; see Figure~\ref{f:asigmaa}.
The conditions above and the fact there are no crossings imply that no
vertex in $(\sigma_G(a),a)$ has an arc with a vertex in $(a,\sigma_G(a))$
or with $a$ or $\sigma_G(a)$, so these vertices are connected only
amongst themselves. It follows that they do not lie in $G_i$ and we
see that (II)(b) holds. Thus $G$ is a labelled $m$-edge-coloured forest satisfying
(II).
\begin{figure}
\begin{center}
\psfragscanon
\psfrag{a=a1}{$\mathbf{a}=\mathbf{a_1}$}
\psfrag{a2}{$\mathbf{a_2}$}
\psfrag{a3}{$\mathbf{a_3}$}
\psfrag{ap}{$\mathbf{a_p}$}
\psfrag{S1}{$S_1$}
\psfrag{Sm}{$S_m$}
\psfrag{Si1}{$S_{i_1}$}
\psfrag{Si2}{$S_{i_2}$}
\psfrag{Si3}{$S_{i_3}$}
\psfrag{Sip}{$S_{i_p}$}
\psfrag{Sip1}{$S_{i_{p-1}}$}
\psfrag{sa=ap}{$\sigma(\mathbf{a})=\mathbf{a_p}$}
\includegraphics[width=6cm]{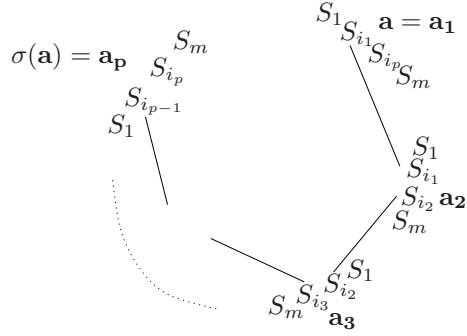} 
\end{center}
\caption{The part of $\Sigma$ between $a$ and $\sigma_G(a)$.}
\label{f:asigmaa}
\end{figure}

Conversely, suppose that we have a labelled $m$-edge-coloured forest
with $k$ vertices satisfying (II).
Let $\Sigma$ be the RNA $m$-diagram of degree $k$
with an arc joining $S_r$ in vertex $i$ with $S_r$ in vertex $j$ if and
only if there is an edge in $G$ between vertices $i$ and $j$ coloured
with symbol $S_r$. We must check that $\Sigma$ can be drawn with no
crossing arcs, i.e.\ that it is noncrossing.

We do this by induction on the number of vertices. Suppose first that
$G$ has more than one connected component, i.e.\ that $G$ is not connected.
By induction, each component $G_i$ corresponds to a noncrossing RNA
$m$-diagram (on the vertices of $G_i$).

Suppose that we had $a_1>a_2\in G_i$ and
$b_1>b_2\in G_j$ for two distinct components $G_i$ and $G_j$, with
arcs between $a_1$ and $a_2$ and $b_1$ and $b_2$ which cross
in $\Sigma$. Then, going around the circle anticlockwise, starting
at vertex $k$, we must encounter
$a_1,b_1,a_2,b_2$ in order, or $b_1,a_1,b_2,a_2$ in order. Swapping $G_i$ and
$G_j$ if necessary, we can assume we are in the first case. But then
$a_1>b_1>a_2>b_2$, contradicting (II)(a). Hence $\Sigma$ is noncrossing.

So we are reduced to the case in which $G$ has exactly one connected
component, i.e.\ $G$ is connected. Suppose that vertex $k$ is incident with
edges $e_1,e_2,\ldots ,e_d$ in $G$, coloured with symbols
$S_{r_1},S_{r_2},\ldots ,S_{r_d}$ where
$r_1<r_2<\ldots <r_d$. Let the endpoints of these edges (other than $k$)
be $v_1,v_2,\ldots ,v_d$. Removing vertex $k$ from $G$ leaves precisely
$d$ trees $T_1,T_2,\ldots ,T_d$ containing vertices $v_1,v_2,\ldots ,v_d$
respectively. By (b), we know that $\sigma_G=S_mS_{m-1}\cdots S_1$
induces the permutation $(k\ k-1\ \cdots 1)$ on the vertices of $G$.

We apply $\sigma_G$ to vertex $k$, and then repeatedly apply $\sigma_G$.
By its definition, each application of $\sigma_G$ corresponds to following a
certain path through $G$, i.e.\ passing along the edges corresponding to the
symbols in the sequence $S_1,S_2,\ldots ,S_m$ in that order,
when such incident edges exist.
Since the edge $e_1$ has symbol $S_{r_1}$, and no edge incident with $k$
has smaller symbol, it follows that, after the first application of
$\sigma_G$, we obtain vertex $k_1:=k-1$ in $T_1$.

Since $\sigma_G$ is a $k$-cycle, after repeated application of $\sigma_G$,
we must leave $T_1$. Suppose that $k_2$ is the number of the first vertex
reached outside $T_1$. Since $r_2$ is the minimum number of a symbol
adjacent to $k$ greater than $r_1$, $k_2$ will lie in $T_2$.
Repeating this argument,
we will obtain $k>k_1>k_2>\cdots >k_d\geq 1$ such that vertices
$k_{i+1}+1,\ldots ,k_i$ lie in tree $T_i$ for $i=1,2,\ldots ,d-1$.
At the final step, the first vertex reached on leaving $T_d$ must be
$k$. Since $\sigma_G$ is a $k$-cycle, all vertices must have been visited.

Let $k_{d+1}=0$. It follows from the above that tree $T_i$ contains precisely
vertices $k_{i+1}+1,\ldots ,k_i$ for each $i$.
Thus, the numbering of the vertices of $G$ is first the vertices
of $T_d$ in some order, then the vertices of $T_{d-1}$ in some order,
then the vertices of $T_{d-2}$, and so on, ending with the vertices of $T_1$
and then finally $k$. Each $T_i$ will
correspond (by the inductive hypothesis) to a noncrossing RNA $m$-diagram
on its vertices. Thus the vertices $v_1,v_2,\ldots ,v_d$
in $G$ will be numbered in decreasing order. The arcs in $\Sigma$
from $k$ to these vertices are numbered by symbols
$S_{r_1},S_{r_2},\ldots ,S_{r_d},$
respectively, with $r_1<r_2<\cdots <r_d$. It follows these arcs
do not cross each other or any of the other arcs in $\Sigma$.
See Figure~\ref{f:bijectionexample1}(a) for an example,
where $v_1=9$, $r_1=1$, and $v_2=8$, $r_2=4$.
Hence, $\Sigma$ is noncrossing and thus an object in (I) as required.

It is clear that the two maps we have constructed are inverse to each
other, so the theorem is proved.
\end{proof}

The following lemma follows easily from the definitions.

\begin{lemma}
A noncrossing RNA $m$-diagram is connected
if and only if the corresponding labelled $m$-edge-coloured forest is connected,
i.e.\ is a tree.
\end{lemma}

\begin{remark} \label{r:sigmaisminus1}
\begin{enumerate}
\item[(1)] Since a forest on $k$ vertices is connected if and only if it has
exactly $k-1$ edges, a noncrossing RNA $m$-diagram with $k$ vertices is
connected if and only if it has $k-1$ arcs.

\item[(2)] In the connected case, the circular order of $G$
is just the permutation $(k\ k-1\ \cdots \ 1)$ and we have
a bijection between the following sets:
\begin{center}
The set of connected noncrossing RNA $m$-diagrams of degree $k$.

$\updownarrow$

The set of labelled $m$-edge-coloured trees with $k$ vertices and circular order
$(k\ k-1\ \cdots \ 1)$.
\end{center}
\item[(3)] The labelling on the vertices for a tree in the latter set in (2)
is determined by a distinguished vertex, that labelled $k$, say, since
$\sigma_G$ then determines the labels on all the other vertices.
\end{enumerate}
\end{remark}

By Remark~\ref{r:sigmaisminus1} and Lemma~\ref{l:alwaysacycle} we have:

\begin{corollary} \label{c:connectedcase}
The bijection in Theorem~\ref{t:pseudoknotbijection} induces a bijection
between the following sets:
\begin{enumerate}
\item[(a)] The set of connected noncrossing RNA $m$-diagrams of degree $k$.
\item[(b)] The set of rooted $m$-edge-coloured trees with $k$ vertices.
\item[(c)] The set of labelled $m$-edge-coloured trees with $k$ vertices
and circular order $(k\ k-1\ \cdots\ 1)$.
\end{enumerate}
\end{corollary}

See Figure~\ref{f:bijectionexample1}(a)-(c) for an illustration of this
bijection. The following is well-known; see, for example,~\cite[Sect.\ 6.2]{stanley99}
(adding the labelling is straightforward).

\begin{figure}
\centering
\psfragscanon
\psfrag{S1}{$S_1$}
\psfrag{S2}{$S_2$}
\psfrag{S3}{$S_3$}
\psfrag{S4}{$S_4$}
\subfigure[A connected noncrossing RNA $4$-diagram of degree $10$.]{\includegraphics[scale=0.5]{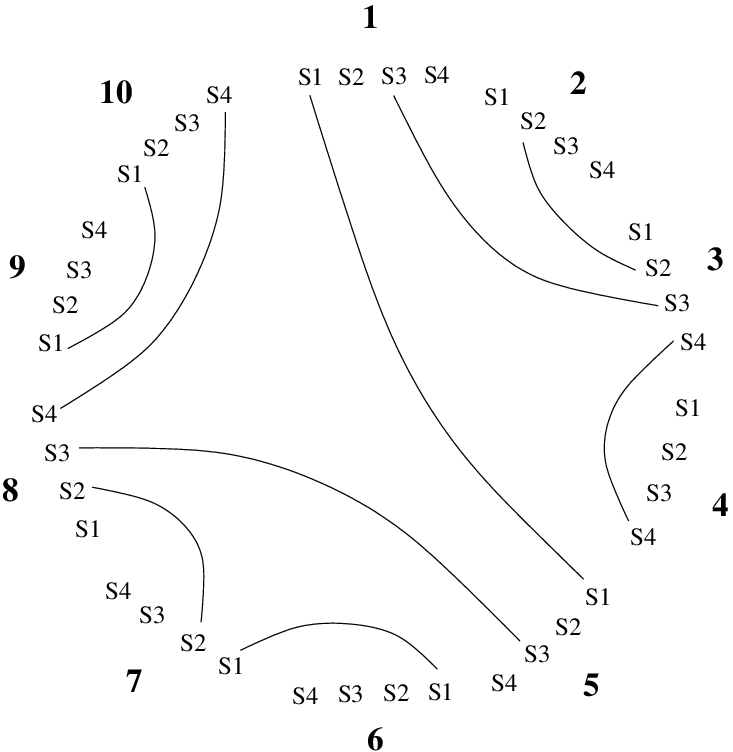}} \\
\subfigure[A rooted $4$-edge-coloured tree with $10$ vertices.]{\includegraphics[scale=0.4]{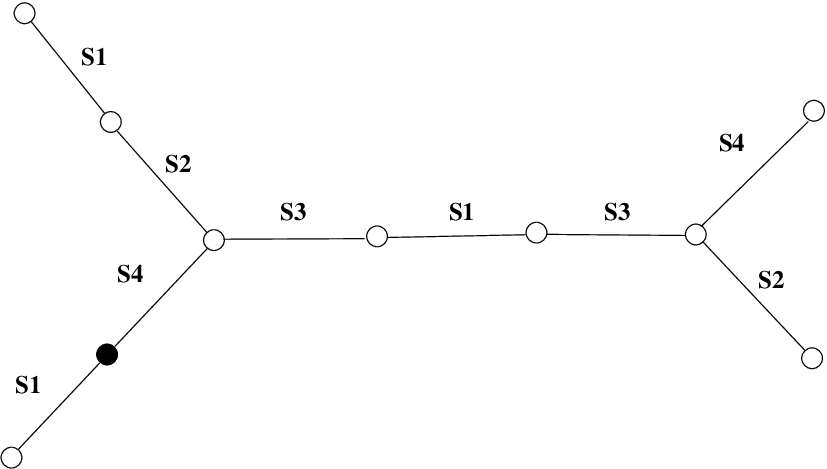}} \qquad
\subfigure[A labelled $4$-edge-coloured tree with $10$ vertices and circular
order $(10\ 9\ \cdots 1)$.]{\includegraphics[scale=0.4]{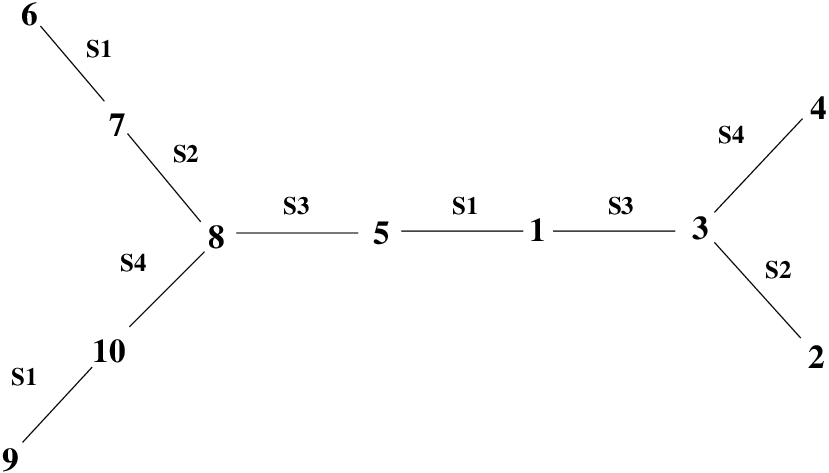}}
\caption{Objects corresponding to each other under the bijections in Corollary~\ref{c:connectedcase}.}
\label{f:bijectionexample1}
\end{figure}

\begin{theorem} \label{t:bijectionshapesdissections}
There is a bijection between the following sets:
\begin{enumerate}
\item[(a)] The set of $m$-edge-coloured trees with $k$ vertices.
\item[(b)] The set of diagonal-coloured $m$-angulations of
an $(m-2)k+2$-sided regular polygon up to rotation.
\end{enumerate}
\end{theorem}

\begin{corollary} \label{c:dissectionbijection}
There is a bijection between the following sets:
\begin{enumerate}
\item[(a)] The set of labelled $m$-edge-coloured trees
with $k$ vertices and circular order $(k\ k-1\ \cdots\ 1)$.
\item[(b)] The set of rooted diagonal-coloured $m$-angulations
of an $(m-2)k+2$-sided regular polygon up to rotation.
\end{enumerate}
\end{corollary}

\begin{proof}
This follows from Theorem~\ref{t:bijectionshapesdissections} and
Corollary~\ref{c:connectedcase}.
\end{proof}

\begin{corollary} \label{c:tkformula}
The cardinality of each of the following sets:
\begin{enumerate}
\item[(a)]
The set of labelled $m$-edge-coloured trees with $k$ vertices and
circular order $(k\ k-1\ \cdots\ 1)$;
\item[(b)]
The set of connected noncrossing RNA $m$-diagrams of degree $k$;
\end{enumerate}
is given by $$T_{k,m}=\frac{m}{(m-2)k+2} \binom{(m-1)k}{k-1}.$$
\end{corollary}

\begin{proof}
Both sets have the same cardinality by Corollary~\ref{c:connectedcase}
and, by Corollary~\ref{c:dissectionbijection}, they have the same
cardinality as the set of rooted diagonal-coloured $m$-angulations
of an $(m-2)k+2$-sided regular polygon up to rotation.
The number $S_{k,m}$ of such $m$-angulations without a root, with no
labelling of diagonals and ignoring rotational equivalence is well-known (see
e.g.~\cite{hiltonpedersen91}).
Let $C_k^m$ be the $k$th Fuss-Catalan number of degree $m$:
$$C_k^m=\frac{1}{k}\binom{mk}{k-1}=\frac{1}{(m-1)k+1}\binom{mk}{k}.$$
Then
$$S_{k,m}=C_k^{m-1}=\frac{1}{(m-2)k+1} \binom{(m-1)k}{k}.$$

Since there are $k$ $m$-sided polygons in an $m$-angulation, there are $k$
possibilities for the root. There are $m$ possibilities for a labelling
since once one diagonal is coloured, all other diagonals in the $m$-angulation
have determined colours using the rule that each $m$-gon must have its
edges coloured $S_1,S_2,\ldots ,S_m$ clockwise around the boundary.
Each orbit of diagonal-coloured rooted $m$-angulations under the action of
the rotation group of the polygon contains $(m-2)k+2$ elements
(the number of sides of $P$).
Hence, we have:
$$T_{k,m}=\frac{kmS_{k,m}}{(m-2)k+2}$$
and the result follows.
\end{proof}

It is interesting to note that by~\cite[\S3]{hoggattbicknell76}, the sequence
$T_{1,m}$, $T_{2,m},\ldots $ is the $m$-fold convolution of the sequence
$S_{0,m}$, $S_{1,m},\ldots $.

\begin{example}
For $m=3,4,5,6$, the first few values of $T_{k,m}$ are given in the
following table:
\begin{center}
\begin{tabular}{|c|c|c|c|c|c|c|c|}
\hline
$k$ & $0$ & $1$ & $2$ & $3$ & $4$ & $5$ & $6$ \\
\hline
$T_{k,3}$ & $1$ & $1$ & $3$ & $9$ & $28$ & $90$ & $297$ \\
\hline
$T_{k,4}$ & $1$ & $1$ & $4$ & $18$ & $88$ & $455$ & $2448$ \\
\hline
$T_{k,5}$ & $1$ & $1$ & $5$ & $30$ & $200$ & $1425$ & $10626$ \\
\hline
$T_{k,6}$ & $1$ & $1$ & $6$ & $45$ & $380$ & $3450$ & $32886$ \\
\hline
\end{tabular}
\end{center}
\vskip 0.3cm
The cases $m=3,4$ are sequences A071724 and A006229, respectively,
in~\cite{sloane10}; the cases $m=5,6$ do not appear. The case $m=3$
appears in~\cite[Prop.\ 7.5]{cfz11}.
\end{example}

\begin{corollary}
The total number of labelled $m$-edge-coloured trees with $k$ vertices
is:
$$U_{k,m}=\frac{m((m-1)k)!}{((m-2)k+2)!}$$
\end{corollary}

\begin{proof}
By Lemma~\ref{l:alwaysacycle}, the circular order of any labelled
$m$-edge-coloured tree with $k$ vertices is a $k$-cycle. By
Corollary~\ref{c:tkformula} the number of such trees with a given circular
order is $T_{k,m}$. Since any $k$-cycle can arise, we have
$U_{k,m}=\frac{1}{(k-1)!}T_{k,m}$, giving the result.
\end{proof}

Note that this result is
already known~\cite{ckss04},~\cite[5.28, p124]{stanley99}.

\section{generalized Induction}
\label{s:induction}

In this section we give the definition of a generalized induction
on labelled $m$-edge-coloured trees with $k$ vertices.
Generalized induction
generates new labelled $m$-edge-coloured trees with the same number
of vertices starting with a given such tree, and the transitive
closure is an equivalence relation. We show that the circular order
is an invariant, giving rise to a classification of the equivalence
classes by $k$-cycles in the symmetric group of degree $k$.

Given a labelled $m$-edge-coloured tree $G$ and integers
$i, j \in \{1, \ldots, m \}$ we define a maximal
$S_i$ - $S_j$ chain $B$ in $G$ to be a (linear)
subtree of $G$ whose edges are only coloured $S_i$ and $S_j$
such that no other edges incident to $B$ are coloured by $S_i$ or $S_j$.

\begin{definition}
Let $G$ be a labelled $m$-edge-coloured tree with $k$ vertices. Fix
$i, j \in \{1, \ldots, m \}$  with $i < j$. Let $B$ be a maximal
$S_i - S_j$ chain in $G$. Define $R^B_{i, j}(G)$ to be the labelled
$m$-edge-coloured tree with $k$ vertices obtained from $G$ by
\begin{itemize}
\item first removing all subtrees in the complement of the maximal
chain $B$
\item interchanging the vertices of each edge of $B$
coloured by $S_j$
\item interchanging the symbols $S_i$ and $S_j$ on the whole maximal
chain $B$
\item reattaching the previously removed subtrees to $B$ at the vertices with the same label they
were removed from. \end{itemize}

Similarly, define $L^B_{i,j}(G)$, where in the second bullet point
in the above definition we interchange the vertices of each edge coloured
by $S_i$ rather than those coloured by $S_j$. We also set
$R_i^B:= R^B_{i,i+1}$ and $L^B_i:=L^B_{i, i+1}$ and we will write
$R_i$ and $L_i$ if $B$ is clear from the context.
\end{definition}

\begin{remark}
\begin{enumerate}
\item[(1)]Induction can also be defined on $m$-edge-coloured trees:
For an $m$-edge-coloured tree with $k$ vertices, choose an
arbitrary vertex-labelling, apply induction, and then remove the vertex
labelling. It is clear that this is independent of the
vertex-labelling chosen.
\item[(2)]The inductions $R^B_{i,j}$ and $L^B_{i,j}$ are mutually
inverse maps.
\end{enumerate}
\end{remark}

\begin{lemma}\label{LemmaB}
Let  $i, j \in \{1, \ldots, m \}$ with $i<j$ and $B$ be a maximal
$S_i-S_j$-chain with no incident edges coloured by $S_{i+1}, \ldots, S_{j-1}$.
Then we have
\begin{align*}
R_{i,j}^B(G) &= L_iL_{i+1}\cdots L_{j-2}R_{j-1}R_{j-2}R_{j-3}\cdots R_i \\
&= R_{j-1}R_{j-2}\cdots R_{i+1}R_iL_{i+1}L_{i+2}\cdots L_{j-1}; \\
L_{i,j}^B(G) &= L_iL_{i+1}\cdots L_{j-2}L_{j-1}R_{j-2}R_{j-3}\cdots R_i \\
&= R_{j-1}R_{j-2}\cdots R_{i+1}L_iL_{i+1}L_{i+2}\cdots L_{j-1},
\end{align*}
where, in each case, the inductions of form $R_p$ and $L_p$ are applied to
all maximal chains contained in $B$.
In particular, the induction $R^B_{i, j}(G)$ can be written as a product of
inductions of the form $R_p$ or $L_p$ for $p=i, i+1, \ldots, j-1$.
\end{lemma}

\begin{proof}
Suppose first that $B$ has the following form:
$$\xymatrix{a_1 \ar@{-}[r]^{S_i} & a_2 \ar@{-}[r]^{S_j} & a_3
\ar@{-}[r]^{S_i} & a_4 & \cdots & a_{r-1} \ar@{-}[r]^{S_i} & a_r}.$$
Then, in $R^B_{i,j}(G)$, $B'$ becomes the chain
$$B'= \xymatrix{a_1 \ar@{-}[r]^{S_j} & a_3 \ar@{-}[r]^{S_i} & a_2
\ar@{-}[r]^{S_j} & a_5 \ar@{-}[r]^{S_i} & a_4 & \cdots & a_{r-2}
\ar@{-}[r]^{S_j} & a_r}.$$
On the other hand, applying $R_{j-2}R_{j-3}\cdots R_{i+1}R_i$ (with
each induction applying to all the maximal chains of the appropriate type
contained in $B$), we obtain the maximal $S_{j-1}-S_j$-chain
$$\xymatrix{a_1
\ar@{-}[r]^{S_{j-1}} & a_2 \ar@{-}[r]^{S_j} & a_3
\ar@{-}[r]^{S_{j-1}} & a_4 & \cdots & a_{r-1} \ar@{-}[r]^{S_{j-1}} &
a_r}.$$
Next, apply $R_{j-1}^B$ to get the maximal $S_{j-1}-S_j$ chain:
$$\xymatrix{a_1
\ar@{-}[r]^{S_{j}} & a_3 \ar@{-}[r]^{S_{j-1}} & a_2
\ar@{-}[r]^{S_{j}} & a_5 & \cdots & a_{r-2} \ar@{-}[r]^{S_{j}} &
a_r}.$$
It is then clear that subsequently applying
$L_iL_{i+1}\cdots L_{j-2}$
(in each case applying the induction to all
the maximal chains of appropriate type contained in the full
subgraph on the vertices $a_1,a_2,\ldots ,a_r$) gives the chain $B'$.

The proof works in a similar way for the other configurations of
maximal $S_i-S_j$-chains, i.e.\ for the above case with $S_i$ and
$S_j$ switched and for the chains of the form  $$\xymatrix{a_1
\ar@{-}[r]^{S_i} & a_2 \ar@{-}[r]^{S_j} & a_3 \ar@{-}[r]^{S_i} & a_4
& \cdots & a_{r-1} \ar@{-}[r]^{S_j} &
a_r}$$ and  $$\xymatrix{a_1 \ar@{-}[r]^{S_j} & a_2 \ar@{-}[r]^{S_i}
& a_3 \ar@{-}[r]^{S_j} & a_4 & \cdots & a_{r-1} \ar@{-}[r]^{S_j} &
a_r}.$$
The other identities are proved similarly.
\end{proof}

\begin{definition}
Call two labelled $m$-edge-coloured trees with $k$ vertices $G,G'$
\emph{induction equivalent} if there is a
sequence of inductions taking $G$ to $G'$, either of the form
$L_i$, for $1 \leq i \leq m-1$ or of the form $R_i$, for $1 \leq i \leq m-1$.
Clearly this is a reflexive relation. It is symmetric since $L_i$ and
$R_i$ are inverse maps and it is easy to see that it is transitive.
Hence it is an equivalence relation. We write $\Gamma(G)$ for
the equivalence class containing $G$.
\end{definition}

We next show that two labelled
$m$-edge-coloured trees with $k$ vertices are induction equivalent if
and only if they have the same circular order.
Firstly, Lemmas~\ref{InvarianceGeneralInduction} and ~\ref{OldLemma 3.3}
show that the circular order is invariant under an induction of the form
$R_i$ or $L_i$. Key to showing the converse is
Proposition~\ref{PropC}, which shows that every labelled $m$-edge
coloured tree is induction equivalent to one whose labels are only
$S_1$ and $S_m$, which is necessarily a line.
This, together with a calculation
of the order of induction (Lemma~\ref{OrderofInduction}) allows
a counting argument for proving the converse; see
Theorem~\ref{InductionCircularOrder}.

\begin{lemma}\label{InvarianceGeneralInduction}
Let $G$ be a labelled $m$-edge-coloured tree with $k$ vertices containing
a maximal $S_i$-$S_j$ chain $B$ with no incident edges coloured
$S_l$, $i< l <j$. Then the circular order, $\sigma_G$, is unchanged after
an induction of the form $R^B_{i,j}$ or $L^B_{i,j}$ is applied.
\end{lemma}

\begin{proof}
We show the result for $R_{i,j}$-induction (the proof for $L_{i,j}$-induction follows a similar
argument).

Let $G$ be a labelled $m$-edge-coloured tree with $k$ vertices and
circular order $\sigma_G$ and containing a maximal $S_i$-$S_j$-chain
$B$ such that no edges coloured $S_l$, for $i<l<j$ are incident
to $B$. Let $G' = R_{i,j}^B(G)$ with circular order $\sigma_{G'}$.
Let $a$ be a vertex in $B$. Let $S'_1, \ldots S'_m$
denote the maps corresponding to the symbols $S_1,\ldots ,S_m$ in
the labelled $m$-edge-coloured tree $G'$.
There are two possible situations to consider.

\emph{Case 1}:
Suppose first that $a$ has edges incident to it which are coloured with
labels $S_t$ with $1\leq t<i$. Consider the edge $e$ incident with
$a$ coloured with label $S_r$ with $r$ minimal.
Let $T$ be the subtree of $G\setminus B$ connected to $a$ via $e$.
Then $\sigma_G(a)$ lies in $T$.
Applying the induction $R_{i,j}^B$ to $G$ results in an $m$-edge-coloured tree
$G'$ in which $a$ is reconnected to $T$ by edge $e$, still coloured
with $S_r$. We see that $\sigma_G(a) = \sigma_{G'}(a)$.

\emph{Case 2}: Suppose that $a$ is not incident with any edges coloured
with $S_l$ for $l<i$.
An easy case-by-case check verifies that $S'_j S'_i (a) = S_j S_i(a)$,
so $S'_j\cdots S'_1(a)=S_j\cdots S_1(a)$.
Any edges incident with $S_jS_i(a)$ with label $S_l$ for $l>j$ are
reattached to the same vertex in $G'$ with the same label.
It follows that $\sigma_G(a) = \sigma_{G'}(a)$.
\end{proof}

\begin{corollary}\label{OldLemma 3.3}
The circular orders of two induction equivalent labelled $m$-edge-coloured
trees with $k$ vertices are the same.
\end{corollary}

\begin{proof}
In Lemma~\ref{InvarianceGeneralInduction} assume that $j=i+1$.
\end{proof}

\begin{remark}
In general, the circular order of a labelled $m$-edge-coloured tree with
$k$ vertices is not invariant under general inductions of the form
$R_{i,j}$ and $L_{i,j}$ with $i<j+1$: in the above proof suppose that we have
an edge $e$ coloured $S_l$, for
$i < l <j$ connected to $S_i(a)$ in the chain $B$ of $G$ and
let $T$ be a subtree of $G$ connected to $e$. Then $\sigma_G(a)$
lies on the subtree $T$. On the other hand in $R^B_{i,j}(G)$, the
edge $S_k$ is connected to $S_i(a) = S'_i S'_j S'_i(a)$ and thus
$\sigma_{G'}(a) \neq \sigma_G(a)$ in general. For an example of this with
$i=1$ and $j=3$ (and $T$ just consisting of the vertex $6$),
see Figure~\ref{f:circularorderbreaks}.

Note that this implies that it is not possible, in general, to write
$R_{i,j}$ as a composition of inductions of the form $R_i$ and $L_i$,
since such inductions preserve the circular order (by
Lemma~\ref{InvarianceGeneralInduction}).
However,  Lemma~\ref{LemmaB} says that sometimes this is possible,
however.
\end{remark}

\begin{figure}
\begin{center}
\psfragscanon
\psfrag{1}{$1$}
\psfrag{2}{$2$}
\psfrag{3}{$3$}
\psfrag{4}{$4$}
\psfrag{5}{$5$}
\psfrag{6}{$6$}
\psfrag{S1}{$S_1$}
\psfrag{S2}{$S_2$}
\psfrag{S3}{$S_3$}
\psfrag{a}{$a$}
\psfrag{R}{$R_{1,3}$}
\psfrag{G}{$G$}
\psfrag{Gp}{$G'$}
\includegraphics[width=6cm]{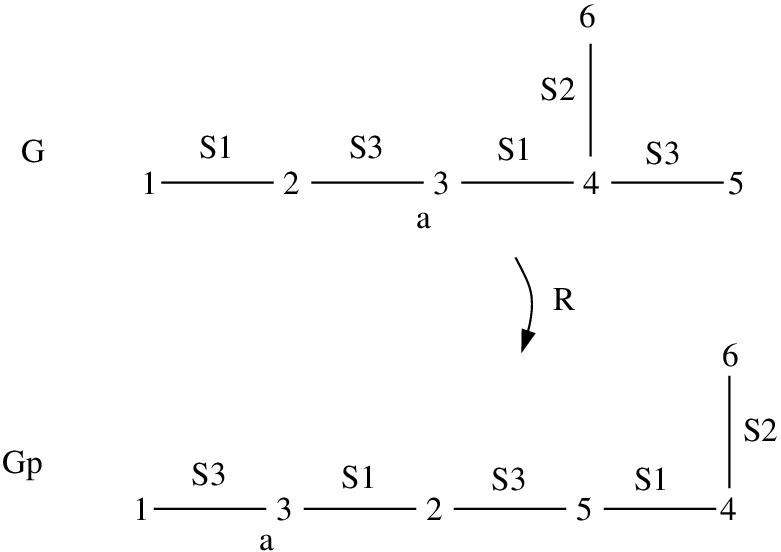}
\end{center}
\caption{The induction $R_{13}^{2,3,4,5}$ does not preserve the circular order: $\sigma_G(3)=6$
while $\sigma_{G'}(3)=5$}
\label{f:circularorderbreaks}
\end{figure}

\begin{prop}\label{PropC}
Let $G$ be a labelled $m$-edge-coloured tree with $k$ vertices.
Then there exists a labelled $m$-edge-coloured tree, induction
equivalent to $G$, whose edges are only coloured by $S_1$ and $S_m$.
\end{prop}

\begin{proof}
We first show that there exists a labelled $m$-edge-coloured tree
$G_2$ with $k$ vertices that is induction equivalent to $G$, none of
whose edges is coloured with $S_2$, by removing the symbols $S_2$
one by one. Firstly, remove all edges coloured with symbols $S_4,
\ldots, S_m$ and call the resulting labelled $3$-edge-coloured tree
$\widetilde{G}$. By \cite[Prop.\ 5.2]{cfz11} there exists a sequence
of inductions of the form $R_1$ and $L_2$ taking
$\widetilde{G}$ to a labelled $3$-edge-coloured tree
$\widetilde{G}_2$ with no edge coloured $S_2$. Let $G_2$ be the
labelled $m$-edge-coloured tree $\widetilde{G}_2$ with the detached
edges reattached (to the vertices with the same label). Since none
of the detached edges are coloured with $S_1, S_2,$ or $S_3$ this
sequence of inductions also takes $G$ to $G_2$ by identifying
maximal chains in $\widetilde{G}$ with corresponding maximal chains
in $G$.

Suppose we have shown that $G$ is induction equivalent to $G_{l-1}$,
where $G_{l-1}$ has no edges coloured $S_2, \ldots, S_{l-1}$.

Then detach all edges coloured $S_{l+2}, \ldots, S_m$ from $G_{l-1}$.
Call the resulting labelled $m$-edge-coloured tree $\widetilde{G}$.
By \cite[Prop.\ 5.2]{cfz11} there is a sequence of inductions of the
form $R_{1,l}$ and $L_l=L_{{l}, {l+1}}$ taking $\widetilde{G}$ to
$\widetilde{G}_l$, where $\widetilde{G}_l$ has no edges coloured
$S_2, \ldots, S_l$. Reattach the detached edges and call the
resulting labelled $m$-edge-coloured tree $G_l$. Since none of the
reattached edges are coloured by $S_1, S_l,$ or $S_{l+1}$, this
sequence of inductions takes $G_{l-1}$ to $G_l$ by identifying the
maximal chains in $\widetilde{G}$ with the maximal chains in
$G_{l-1}$. By Lemma~\ref{LemmaB}, each application of $R_{1,l}$ can be written
as a composition of inductions $R_p$, $L_p$, $p=1,2,\ldots ,l-1$,
since $G_{l-1}$ has no edges coloured $S_2,\ldots ,S_{l-1}$.

Note that none of the symbols $S_2, \ldots, S_l$ appears
in $G_l$. Hence, by induction on $l$, we can construct $G_{m-1} =
G'$ with no symbols $S_2, \ldots, S_{m-1}$ and a sequence of
inductions, each of the form $R_p$ or $L_p$, with $1\leq p \leq m-1$,
taking $G$ to $G'$.
\end{proof}

\begin{lemma}\label{OrderofInduction}
\begin{enumerate}
\item[(a)]
Let $k$ be odd, let $i \neq j$ and let $G$ be a labelled $m$-edge-coloured
tree with $k$ vertices whose edges are coloured with $S_i$ and $S_j$ only.
Let $\shape$ be the underlying unlabelled tree of $G$.
Then $R_{i,j}$ and $L_{i,j}$ have order $k$ on $G$, producing $k$ distinct
trees with shape $\shape$.
\item[(b)]
Let $k$ be even, $i \neq j$ and let $G$ be a labelled $m$-edge-coloured
tree with $k$ vertices and whose edges are coloured with
 $S_i$ and $S_j$ only. Let $\shape$ be the underlying unlabelled tree
of $G$.
Then $R_{i,j}$ and $L_{i,j}$ have order $k$ on $G$, producing $k/2$
distinct labelled trees whose underlying unlabelled tree is $\shape$ in
each case and
$k/2$ distinct labelled trees whose underlying unlabelled tree is
$R_{i,j}(\shape)$.
\end{enumerate}
\end{lemma}

\begin{proof}
(a)  Since $\shape$ contains only the symbols $S_i$ and
$S_j$, it is a line. Suppose the line is drawn horizontally and
suppose the leftmost edge has colour $S_j$ and vertices $a_1$ and
$a_2$ from left to right. Then in $(R_{i,j}^G)^d(G)$, if it is drawn with
orientation given by $ \xymatrix{ \cdot \ar@{-}[r]^{S_i} & a_1 \ar@{-}[r]^{S_j}
& \cdot}$ (where one of the edges may not exist), the vertex $a_1$ is
the $d^{th}$ vertex from the left. It is clear that all the induced
trees $(R_{i,j}^G)^d(G)$ have the same underlying unlabelled tree (those for
$d$ odd should be read from right to left).

(b) The proof is similar to the one in (a), except that for $d$ odd,
the underlying unlabelled tree of $(R_{i,j}^G)^d(G)$ is $R_{i,j}(\shape)$.
\end{proof}

\begin{remark} \label{RemarkOnOrderofInduction}
Since in the context of Lemma~\ref{OrderofInduction} the underlying unlabelled
tree of $G$ is a line, we can replace $R_{i,j}$ or $L_{i,j}$ in
Lemma~\ref{OrderofInduction} by one of the products in Lemma~\ref{LemmaB}.
\end{remark}

We finally obtain:

\begin{theorem} \label{InductionCircularOrder}
Let $G$ and $G'$ be labelled $m$-edge-coloured trees with $k$ vertices.
Then $G'$ is in $\Gamma(G)$ if and only if $\sigma_G = \sigma_{G'}$.
\end{theorem}

\begin{proof}
Suppose that $G'$ is in $\Gamma(G)$. We have, by Lemma~\ref{OldLemma 3.3},
that $\sigma_G = \sigma_{G'}$. Conversely,
suppose that $\sigma_G = \sigma_{G'}$. By Proposition~\ref{PropC} there exist
labelled $m$-edge-coloured trees with $k$ vertices, $G_*$, in $\Gamma(G)$ and
$G'_*$ in $\Gamma(G')$ where the underlying unlabelled trees of $G_*$ and
$G'_*$ contain symbols $S_1$ and $S_m$ only.

Suppose that $k$ is odd and that $ G_*$ and $G'_*$ are as follows:
$$ G_* : \xymatrix{a_1 \ar@{-}[r]^{S_1} & a_2
\ar@{-}[r]^{S_m} & a_3 \ar@{-}[r]^{S_1} & a_4 & \cdots & a_{k-1}
\ar@{-}[r]^{S_m} & a_k}$$
$$ G'_* : \xymatrix{a'_1 \ar@{-}[r]^{S_1} & a'_2
\ar@{-}[r]^{S_m} & a'_3 \ar@{-}[r]^{S_1} & a'_4 & \cdots & a'_{k-1}
\ar@{-}[r]^{S_m} & a'_k}.$$ Then
$$\sigma_{G_*}= (a_1 a_3 a_5 \ldots a_k a_{k-1} a_{k-4}\ldots a_2) =
 (a'_1 a'_3 a'_5 \ldots a'_k a'_{k-1} a'_{k-4}\ldots a'_2) =
 \sigma_{G'_*}.$$ Because the underlying unlabelled tree of $G'_*$ is a line,
$G'_*$ is determined by  $a'_1$ and its circular order.
Thus, given $G_*$ there are at most $k$ possibilities
 for $G'_*$ such that $\sigma_{G_*} = \sigma_{G'_*}$ holds.

By Lemma~\ref{OrderofInduction} and Lemma~\ref{OldLemma 3.3},
repeatedly applying either $R_{i,j}$-induction or $L_{i,j}$-induction to
$G_*$ gives $k$ distinct labelled $m$-edge-coloured trees $H$ with $k$
vertices satisfying  $\sigma_{G_*} = \sigma_H$.
Hence any $G'_*$ such that  $\sigma_{G_*} = \sigma_{G'_*}$  must be one
 of these trees $H$. Therefore, using Remark~\ref{RemarkOnOrderofInduction},
$G'_*$ is in $\Gamma(G_*)$ and thus  $G'$ is in $\Gamma(G)$.

Suppose that $k$ is even. In this case, there are at most $k/2$
distinct possibilities for $G'_*$ because of the symmetry of the
underlying unlabelled tree. The result then follows from
Lemma~\ref{OrderofInduction}, Remark~\ref{RemarkOnOrderofInduction},
and Lemma~\ref{OldLemma 3.3} as above.
\end{proof}

\begin{lemma} \label{l:OrderofInduction2}
Let $G$ be a labelled $m$-edge-coloured tree with $k$ vertices containing
a maximal $S_i$-$S_j$ chain $B$ with $l$ vertices and no incident edges
coloured $S_l$,
$i<l<j$. Then $R_{i,j}^B$ has order $l$ on $G$. In particular, $R_i$ and
$L_i$ have finite order on any maximal chain in $G$.
\end{lemma}
\begin{proof}
This follows easily from Lemma~\ref{OrderofInduction}.
\end{proof}

Let $G$ be a labelled $3$-edge-coloured tree with $k$ vertices.
Let $\Gamma'(G)$ be the smallest set of labelled $3$-edge-coloured trees
with $k$ vertices closed under $R_1$ and $L_2$
(as defined in~\cite[\S4]{cfz11}, identifying the symbols $+,=,-$ with
$S_1,S_2$ and $S_3$ respectively). Then:

\begin{corollary} \label{l:cfzrelationship}
\begin{enumerate}
\item[(a)] Two labelled $m$-edge-coloured trees with $k$ vertices
are induction equivalent if and only if there is a sequence of inductions
only of the form $R_p$ (respectively, only of the form $L_p$) taking
one to the other.
\item[(b)]
If $G$ is a $3$-edge-coloured tree with $k$ vertices, then
$\Gamma(G)=\Gamma'(G)$.
\end{enumerate}
\end{corollary}
\begin{proof}
Part (a) follows from Lemma~\ref{l:OrderofInduction2}, using the fact
that $L_p$ (on a given maximal chain) is the inverse of $R_p$.
For part (b), note that (on a given maximal chain), $L_1$ is the inverse
of $R_1$ and $R_1$ has finite order. Similarly, $R_2$ is the inverse of $L_2$
and $L_2$ has finite order.
\end{proof}

Finally in this section, we note the interesting fact that every induction
equivalence class of labelled $m$-edge coloured trees contains a labelling
of any given unlabelled $m$-edge coloured tree. We first have:

\begin{lemma} \label{l:allequivalent}
Any two (unlabelled) $m$-edge coloured trees are induction equivalent.
\end{lemma}
\begin{proof}
Let $\shape$ and $\shape'$ be arbitrary $m$-edge-coloured trees with $k$
vertices. By Lemma~\ref{PropC}, $\shape$ is
induction equivalent to a labelled $m$-edge-coloured tree with $k$
vertices containing only the symbols $S_1$ and $S_m$; similarly for $\shape'$.
If $k$ is odd there is only one such tree:
$$\xymatrix{ \bullet \ar@{-}[r]^{S_m} & \bullet
\ar@{-}[r]^{S_1} & \bullet \ar@{-}[r]^{S_m} & \bullet  & \cdots & \bullet
\ar@{-}[r]^{S_1} & \bullet},$$
and it follows that $\shape$ and $\shape'$ are induction
equivalent. If $k$ is even, there are two such trees:
$$ \shape_m : \xymatrix{ \bullet \ar@{-}[r]^{S_m} & \bullet \ar@{-}[r]^{S_1} & \bullet
\ar@{-}[r]^{S_m} & \bullet  & \cdots & \bullet \ar@{-}[r]^{S_m} &
\bullet}$$
$$ \shape_1 : \xymatrix{ \bullet \ar@{-}[r]^{S_1} & \bullet \ar@{-}[r]^{S_m} & \bullet
\ar@{-}[r]^{S_1} & \bullet  & \cdots & \bullet \ar@{-}[r]^{S_1} &
\bullet}.$$
Then $R_{1, m}^{\shape_m}(\shape_m) =
\shape_1$, so $\shape_m$ is induction equivalent to
$\shape_1$ by Lemma~\ref{LemmaB}. It follows that $\shape$ and $\shape'$
are induction equivalent in this case also.
\end{proof}

\begin{corollary}
Let $G$ be a labelled $m$-edge-coloured tree with $k$ vertices. Then every
possible $m$-edge-coloured tree with $k$ vertices appears as the underlying
unlabelled $m$-edge-coloured tree of a labelled $m$-edge-coloured tree in
$\Gamma(G)$.
\end{corollary}
\begin{proof}
Given a labelled $m$-edge-coloured tree with $k$ vertices, $G$, whose
underlying unlabelled tree is $\shape$, and an arbitrary $m$-edge-coloured tree
with $k$ vertices, $\shape'$, Lemma~\ref{l:allequivalent} shows that
$\shape$ and $\shape'$ are induction equivalent. It follows that $G$
and a vertex-labelling of $\shape'$ are induction equivalent,
and the result follows.
\end{proof}


\section{Induction on $m$-Angulations of Polygons}

By Theorem~\ref{t:bijectionshapesdissections}, the set of labelled
$m$-edge-coloured trees with $k$ vertices is in bijection with the set of
$m$-gon-labelled diagonal-coloured $m$-angulations of a polygon $P_n$
with $n=(m-2)k+2$ sides, up to rotation, where $m$-gon-labelled means that
the $m$-gons are labelled with $1,2,\ldots ,k$. The colours
satisfy a \emph{boundary rule}, i.e.\ those bounding each $m$-gon in the
$m$-angulations must be the symbols $S_1, \ldots, S_m$ in clockwise order.

Our aim in this section is to rewrite induction in the language of
$(m-2)$-clusters. The set of $m$-angulations of $P_n$ containing $k$ $m$-gons
is in bijection with the set of $(m-2)$-clusters of type $A_{k-1}$ by
Fomin-Reading~\cite{fominreading05}.
Mutation corresponds to rotating a diagonal one step anticlockwise
in the subpolygon obtained when the diagonal is removed.

We shall see that induction has a description as a composition
of such mutations and appropriate recolourings of diagonals and
vertex relabellings.
An easy induction argument based on cutting an $m$-angulation
along one of its sides shows that:

\begin{lemma}\label{BoundaryEdges}
Every $m$-angulation $\M$ of a polygon has at least two $m$-gons with
$m-1$ boundary edges or is an $m$-angulation of an $m$-gon.$\Box$
\end{lemma}

\begin{remark} \label{RotationSubLemma}
If all the diagonals in an $m$-angulation are incident with one
vertex, then by moving them one after another,
we get the same $m$-angulation rotated through $2\pi/n$.
\end{remark}

We label the vertices of $P_n$ $1,2,\ldots ,n$ clockwise around
the boundary, and use the notation $[i,j]$ to denote a diagonal in the
polygon connecting vertex $i$ with vertex $j$.

\begin{lemma}\label{Rotation Lemma}
Let $\M$ be an $m$-angulation of $P_n$ with $k$ $m$-gons.
Then there is an explicit sequence of mutations taking
$\M$ to its rotation through $2\pi /n$ anticlockwise.
\end{lemma}

\begin{proof}
We induct on the number $k$ of $m$-gons in the $m$-angulation.
The result is trivial if $k=1$, so consider the general case. By
Lemma~\ref{BoundaryEdges}, there is an $m$-gon
$M$ with a unique internal edge $e$, joining vertices $[i, i+(m-1)]$
for some $i$.
Let $R$ be the subpolygon of $P_n$ given by the union of the $m$-gons
incident with $i$. Applying Remark~\ref{RotationSubLemma}
to the induced $m$-angulation of $R$ to rotate it one step anticlockwise,
we obtain a new $m$-angulation of $R$, and hence of $P_n$,
containing an $m$-gon $M'$ with vertices $i-1,i,i+1,\ldots ,i+(m-1)-1$.
Let $R'$ be the subpolygon of $R$ with $M'$
removed. Applying Remark~\ref{RotationSubLemma} repeatedly
we may rotate the induced $m$-angulation of $R'$ one step clockwise.

We obtain a new $m$-angulation of $P_n$. By induction we may rotate the
subpolygon with $M'$ removed anticlockwise by an explicit sequence.
It is easy to check that the total effect of the above is to rotate the
original $m$-angulation of $P_n$ one step anticlockwise.
\end{proof}

\begin{definition}
Let $\M$ be a diagonal-coloured $m$-angulation of $P_n$.
If there is at least one internal diagonal in $\M$ and
all of the internal diagonals of $\M$ are coloured only
with $S_i$ or $S_{i+1}$ for fixed $i$, we call $\M$ a
\emph{snake $m$-angulation}. A subpolygon of $P_n$ with this property is
called a \emph{snake} subpolygon. Note that in any snake subpolygon the
internal diagonals must be of the form $[i_1,i_2]$, $[i_2,i_3]$, and so on.
\end{definition}

\begin{figure}[H]
\begin{center}
\psfragscanon
\psfrag{S1}{$\scriptstyle S_1$}
\psfrag{S2}{$\scriptstyle S_2$}
\psfrag{S3}{$\scriptstyle S_3$}
\psfrag{S4}{$\scriptstyle S_4$}
\psfrag{Step 1}{Step 1}
\psfrag{Step 2}{Step 2}
\psfrag{Redraw}{Redraw}
\includegraphics[height=9cm]{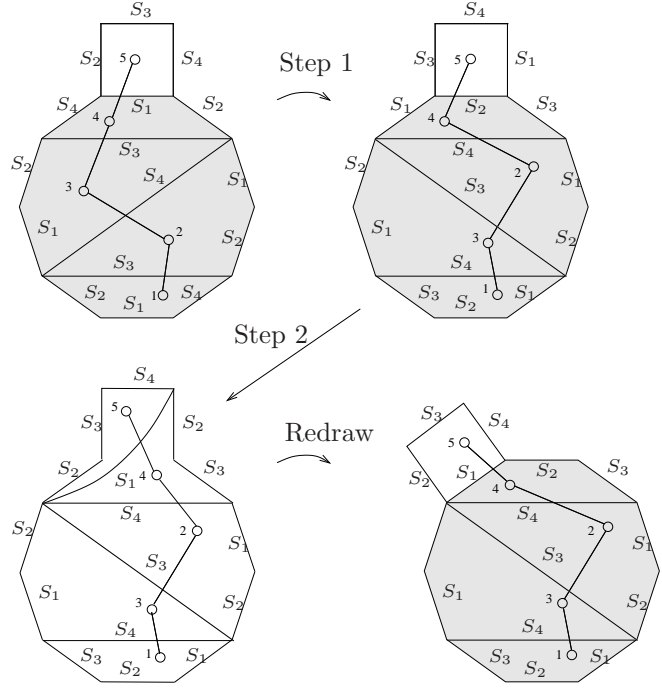} 
\end{center}
\caption{Mutations giving rise to $R_3$ induction (see Proposition~\ref{Induction on Polygon Shapes}).}
\label{f:reattach}
\end{figure}

In a cluster algebra context, snake triangulations
appeared in~\cite[Sect.\ 3.5]{fominzelevinsky03b} and for
general $m$ in~\cite[Sect.\ 5.1]{fominreading05} as
\emph{$m$-snakes}.

Let $\M$ be an $m$-gon-labelled diagonal-coloured $m$-angulation of $P_n$
and let $1\leq i\leq m$. Choose a maximal snake subpolygon $\B$ of $\M$ with
internal diagonals coloured $S_i$ or $S_{i+1}$. Let
${\R}^{\B}_i(\M)$ be defined using the following procedure.

\noindent \textbf{Step 1}: Let $M_1,M_2$ be the two $m$-gons in $\B$ with $m-1$
boundary edges in $\B$, and let $e=e_1$  be the internal
edge of $M_1$ in $\B$. Let $e_2,\ldots ,e_N$ be the other internal edges
in $\B$, numbered so that $e_{i-1}$ and $e_i$ are incident for all $i$.
If $e_1$ has colour $S_i$ (respectively, $S_{i+1}$), mutate edges
$e_2,e_4,\ldots $ (respectively, $e_1,e_3,\ldots $), recolouring
the new diagonals with $S_i$ and using the boundary rule to recolour
the rest of the polygon. The labelling of an $m$-gon is given
by the number of the $m$-gon before mutation whose intersection with the
boundary of $\B$ was the same.

\noindent \textbf{Step 2}: Let $C_1,C_2$ be the union of the connected
components of the complement of $\B$ in $P_n$ incident with $M_1,M_2$
respectively, and
let $D_i = C_i \cup M_i$. Applying Lemma~\ref{Rotation Lemma} to $D_1,D_2$
we get a sequence of mutations rotating the induced $m$-angulation of
each $D_i$ anticlockwise one step around its boundary. We recolour
the diagonals according to the boundary rule. The label of the image of
an $m$-gon under this rotation is the same.

Note that Step 2 has the same effect as detaching each component of $C_i$
in the original $m$-angulation from an edge of $M$ coloured $S_j$,
$j\not=i,i+1$, and reattaching it after Step 1 to a boundary edge of
$M$ with the same symbol $S_j$ (now one step anticlockwise around the boundary
of $M$), keeping the original colouring of the diagonals of $C_i$
(i.e.\ from before Step 1).
Such a boundary edge of $M$ always exists, since $j\not=i,i+1$.

It is easy to see that the above procedure does not depend on the initial
choice of $M_1$. Note also that the procedure commutes with any rotational
symmetry of $P_n$ and so gives a well defined induction on a
diagonal-coloured $m$-angulation of $P_n$ up to rotation.

Comparing the definitions of $R_i^B$ and ${\R}_i^{\B}$ we see that:

\begin{prop} \label{Induction on Polygon Shapes}
Suppose that $1\leq i\leq m$ and let $G$ be a labelled $m$-edge-coloured tree
with $k$ vertices containing a maximal $S_i-S_{i+1}$ chain $B$.
Let $\M$ be the corresponding $m$-gon-labelled
diagonal-coloured $m$-angulation of $P_n$
up to rotation with maximal snake subpolygon $\B$ corresponding to $B$.
Then $R_i^{B}(G)$ corresponds to $R_i^{\B}(\M)$.~$\Box$
\end{prop}

For an example of Proposition~\ref{Induction on Polygon Shapes},
see Figure~\ref{f:reattach}, with the snake subpolygon shaded.

\medskip

\noindent \textbf{Acknowledgements:} We would like to thank Luca Q.
Zamboni for some helpful conversations. RJM would like to thank
Sibylle Schroll and the Department of Mathematics at the University
of Leicester for their kind hospitality.

\bibliographystyle{plain}

\end{document}